\documentclass{amsart}
\usepackage{mathrsfs}
\usepackage{amsmath}
\usepackage{amssymb}

\usepackage[all]{xy}

\DeclareMathOperator{\Supp}{Supp} \DeclareMathOperator{\Pic}{Pic}
\DeclareMathOperator{\codim}{codim} \DeclareMathOperator{\im}{im}
 
\DeclareMathOperator{\texthocolim}{hocolim}
\DeclareMathOperator{\hocolim}{\underrightarrow{\texthocolim}}
\DeclareMathOperator{\textholim}{holim}
\DeclareMathOperator{\holim}{\underleftarrow{\textholim}}
\DeclareMathOperator{\coker}{coker} \DeclareMathOperator{\Ext}{Ext}
\DeclareMathOperator{\Hom}{Hom} \DeclareMathOperator{\Tr}{Tr}

\newtheorem{theorem}{Theorem}[section]
\newtheorem{lemma}[theorem]{Lemma}
\newtheorem{proposition}[theorem]{Proposition}

\theoremstyle{definition}
\newtheorem{definition}[theorem]{Definition}
\newtheorem{example}[theorem]{Example}

\theoremstyle{remark}

\numberwithin{equation}{section}

\begin{document}

\title{On the Generic Vanishing Theorem of Cartier Modules}

\author{Alan Marc Watson}
\address{Department of Mathematics, University of Utah, Salt Lake City, Utah 84102}
\curraddr{155 S 1400 E Room 233, Salt Lake City, Utah 84112}
\email{watson@math.utah.edu}

\author{Yuchen Zhang}
\address{Department of Mathematics, University of Utah, Salt Lake City, Utah 84102}
\email{yzhang@math.utah.edu}


\date{\today}



\begin{abstract}
We generalize the Generic Vanishing theorem by Hacon and Patakfalvi
in the spirit of Pareschi and Popa. We give several examples
illustrating the pathologies appearing in the positive
characteristic setting.
\end{abstract}

\maketitle

\section{Introduction}

Let $X$ be a smooth projective variety, let $a:X\to A$ be the
Albanese morphism and let $$V^i(\omega_X)=\{P\in
\Pic^0(X)|H^i(X,\omega_X\otimes P)\neq 0\}$$ be the cohomology
support loci. In \cite{GL90} and \cite{GL91}, Green and Lazarsfeld
proved the following theorem which is an essential result in the
study of irregular varieties (see, for example,
\cite{F09},\cite{JLT}and \cite{S93}).

\begin{theorem}
Let $X$ be a smooth complex projective variety. Then every
irreducible component of $V^i(\omega_X)$ is a translate of a
subtorus of $\Pic^0(X)$ of codimension at least $$i-\dim X+\dim
a(X).$$ If $X$ has maximal Albanese dimension, then there are
inclusions:
$$V^0(\omega_X)\supset V^1(\omega_X)\supset \cdots \supset V^{\dim
X}(\omega_X)=\{\mathcal{O}_X\}.$$
\end{theorem}

The theorem was first proven using Hodge theory. An alternative
point of view using the Fourier-Mukai transforms $R\hat{S}$ and $RS$
emerged in \cite{H04} and \cite{PP11}. Specifically, in \cite{PP11},
Pareschi and Popa proved the following theorem.

\begin{theorem}\label{main_PP11}
Let $A$ be an abelian variety. Let $\mathcal{F}$ be a coherent sheaf
on $A$. The following are equivalent:
\begin{enumerate}
\item For any sufficiently ample line bundle $L$ on $\hat{A}$,
$H^i(A, \mathcal{F}\otimes \hat{L}^\vee)=0$ for any $i>0$, where
$\hat{L}=R^0\hat{S}(L)=R\hat{S}(L)$,
\item $R^i\hat{S}(D_A(\mathcal{F}))=0$ for any $i\neq 0$,
\item $\codim \Supp R^i\hat{S}(\mathcal{F})\geqslant i$ for any $i\geqslant
0$, and
\item $\codim V^i(\mathcal{F})\geqslant i$ for any $i\geqslant 0$.
\end{enumerate}
\end{theorem}

The theorem holds even in positive characteristic. But, in order to
apply it to the canonical bundle of irregular varieties via Albanese
maps, we need the result of Koll\'{a}r in \cite{K86I} and
\cite{K86II} or Grauert-Reimanschneider Vanishing which is known to
fail in positive characteristic (see \cite{HK}). In \cite{HP}, Hacon
and Patakfalvi suggested that, instead of $R^ia_*\omega_X$, we
should consider the inverse limit of the push-forwards of
$R^ia_*\omega_X$ by the Frobenius map. In particular, they proved
the the following results.

\begin{theorem}\label{main_HP13}
Let $k$ be an algebraically closed field of characteristic $p>0$ and
$A$ be an abelian variety over $k$. Let $\{\Omega_e\}$ be an inverse
system of coherent sheaves on $A$ such that for any sufficiently
ample line bundle $L$ on $\hat{A}$ and any $e\gg 0$,
$H^i(A,\Omega_e\otimes \hat{L}^\vee)=0$ for all $i>0$. Then, the
complex $$\Lambda=\hocolim R\hat{S}(D_A(\Omega_e))$$ is a
quasi-coherent sheaf in degree 0, i.e.,
$\Lambda=\mathcal{H}^0(\Lambda)$. Here $\hocolim$ is a
generalization of direct limit to derived category (see Section
\ref{sec_derived_category}).
\end{theorem}

\begin{theorem}
If $\{\Omega_e\}$ is a Cartier module, then it satisfies the
condition in Theorem \ref{main_HP13}. In particular, let $X$ be a
normal, projective variety over an algebraically closed field $k$ of
characteristic $p>0$, then $\Omega_e=F^e_*S^0a_*\omega_X$ satisfies
the condition in Theorem \ref{main_HP13}.
\end{theorem}

One should regard Theorem \ref{main_HP13} as a generalization of
$(1)\Rightarrow (2)$ in Theorem \ref{main_PP11}. It is a natural
question to ask what is the appropriate generalization of the
statements for $(3)$ and $(4)$ in Theorem \ref{main_PP11} to the
positive characteristic setting and if all the resulting conditions
are equivalent to each other.

In this paper, we generalize Hacon and Patakfalvi's theorem as
follows.

\begin{theorem}\label{main}
(See Theorem \ref{WIT=limKV}, \ref{WITtoGV} and \ref{GVtolimKV}) Let
$A$ be an abelian variety. Let $\{\Omega_e\}$ be an inverse system
of coherent sheaves on $A$ satisfying the Mittag-Leffler condition
and let $\Omega=\varprojlim \Omega_e$. Let
$\Lambda_e=R\hat{S}(D_A(\Omega_e))$ and $\Lambda=\hocolim\Lambda_e$.
The following are equivalent:

\begin{enumerate}
\item[(1)] For any ample line bundle $L$ on $\hat{A}$, $H^i(A,
\Omega\otimes \hat{L}^\vee)=0$ for any $i>0$.
\item[(1')] For any fixed positive integer $e$ and any $i>0$, the
homomorphism $$H^i(A, \Omega\otimes \hat{L}^\vee) \to H^i(A,
\Omega_e\otimes \hat{L}^\vee)$$ is 0 for any sufficiently ample line
bundle $L$.
\item[(2)] $\mathcal{H}^i(\Lambda)=0$ for any $i\neq 0$.
\end{enumerate}

If any of these conditions is satisfied, then we will call
$\{\Omega_e\}$ a GV-inverse system of coherent sheaves.

These conditions imply the following equivalent conditions:

\begin{enumerate}
\item[(3)] For any scheme-theoretic point $P\in A$, if $\dim P>i$,
then $P$ is not in the support of $$\im (R^i\hat{S}(\Omega) \to
R^i\hat{S}(\Omega_e))$$ for any $e$.
\item[(3')] For any scheme-theoretic point $P\in A$, if $\dim P>i$,
then $P$ is not in the support of $$\im (\varprojlim
R^i\hat{S}(\Omega_e) \to R^i\hat{S}(\Omega_e))$$ for any $e$.
\end{enumerate}

If $\{R^i\hat{S}(\Omega_e)\}$ satisfies the Mittag-Leffler condition
for any $i\geq 0$, then (3) and (3') also imply (1), (1') and (2).
\end{theorem}

We should make a remark that even if $\{\Omega_e\}$ is a Cartier
module, $\{R^i\hat{S}(\Omega_e)\}$ does not necessarily satisfy the
Mittag-Leffler condition (see Example \ref{notML_Car}). We are
unable to prove the equivalence in this case. On the other hand, the
statement about $V^i(\Omega)$ is still missing. We will give an
example (see Example \ref{chain_failure}) where the chain of
inclusions for $V^i(\Omega)$ fails. Since the support of $\im
(R^i\hat{S}(\Omega) \to R^i\hat{S}(\Omega_e))$ is usually not closed
(see Example \ref{notML_notCar}), it is not a good idea to talk
about its codimension.

In a sequence of papers \cite{PPI,PPII,PPIII}, Pareschi and Popa
introduced the notion of M-regularity which parallels and
strengthens the usual Castelnuovo-Mumford regularity with respect to
polarizations on abelian varieties and developed several results on
global generation. In \cite{PPIII}, the following characterization
of M-regularity is given.

\begin{theorem}
Let $A$ be an abelian variety and $\mathcal{F}$ be a coherent sheaf
on $A$ satisfying the Generic Vanishing conditions. The following
conditions are equivalent:
\begin{enumerate}
\item $\mathcal{F}$ is M-regular, i.e., $R^0\hat{S}(D_A(\mathcal{F}))$ is torsion-free.
\item $\codim \Supp R^i\hat{S}(\mathcal{F})> i$ for any $i\geqslant 0$
\item $\codim V^i(\mathcal{F})>i$ for any $i\geqslant 0$.
\end{enumerate}
\end{theorem}

We will generalize the theorem above to inverse systems as follows.

\begin{theorem}
(See Theorem \ref{WITtoGV}) Let $A$ be an abelian variety and
$\{\Omega_e\}$ be a GV-inverse system of coherent sheaves on $A$
such that
\begin{enumerate}
\item $\{\Omega_e\}$ is M-regular in the sense that $\mathcal{H}^0(\Lambda)$ is torsion-free.
\end{enumerate}
Then
\begin{enumerate}
\item[(2)] for any scheme-theoretic point $P\in A$, if $\dim P\geqslant
i$, then $P$ is not in the support of $$\im (R^i\hat{S}(\Omega) \to
R^i\hat{S}(\Omega_e))$$ for any $e$.
\end{enumerate}
\end{theorem}

We are unable to prove the converse statement of the above theorem.

\subsection*{Acknowledgement} The authors would give special thank
to their advisor Christopher Hacon for suggesting this problem and
sharing an early draft of \cite{HP}. We are also in debt to Zsolt
Patakfalvi for useful discussions, especially for pointing out that
Example \ref{exmpl_HP} and Example \ref{chain_failure} are the same
in the supersingular case.

\section{Preliminaries}

We work over a perfect field $k$ of arbitrary characteristic.

\subsection{Derived category}\label{sec_derived_category}
We recall some basic notations in derived category. For details, we
refer to \cite[Section 2.1]{HP} and \cite{N96}.

Given a variety $X$, let $D(X)$ be the derived category of
$\mathcal{O}_X$-modules and $D_{qc}(X)$ (resp. $D_c(X)$) be the full
subcategory consisting of bounded complex whose cohomologies are
quasi-coherent (resp. coherent). For any object $\mathcal{E}\in
D_{qc}(X)$, $\mathcal{E}[n]$ denotes the object obtained by shifting
$\mathcal{E}$, $n$ places to the left, and
$\mathcal{H}^n(\mathcal{E})$ denotes the $n$-th homology of a
complex representing $\mathcal{E}$.

Let $X$ be a variety of dimension $n$ and
$\omega_X^\bullet=p^!\mathcal{O}_k$ denote its dualizing complex
such that $\mathcal{H}^{-n}(\omega_X^\bullet)\cong\omega_X$. The
dualizing functor $D_X$ is defined by
$D_X(\mathcal{E})=R\mathcal{H}om(\mathcal{E},\omega_X^\bullet)$ for
any $\mathcal{E}\in D_{qc}(X)$. We have Grothendieck Duality:

\begin{theorem}[Grothendieck Duality]\label{GroDuality}
Let $f:X\to Y$ be a proper morphism of quasi-projective varieties
over a field $k$. Then
$$Rf_*(D_X(\mathcal{E}))=D_Y(Rf_*(\mathcal{E}))$$ for any
$\mathcal{E}\in D_{qc}(X)$.
\end{theorem}

As a generalization of direct limit in triangulated category, the
homotopy colimit is defined as follows.

\begin{definition}
Let $\{\mathcal{C}_e\}$ be a direct system of objects in
$D_{qc}(X)$,
$$\mathcal{C}_1\xrightarrow{f_1}\mathcal{C}_2\xrightarrow{f_2}\cdots.$$
The homotopy colimit $\hocolim \mathcal{C}_e$ is defined by the
following triangle
$$\bigoplus \mathcal{C}_e\xrightarrow{\text{id}-\bigoplus f_e} \bigoplus
\mathcal{C}_e \to \hocolim C_e \to \bigoplus \mathcal{C}_e[1].$$
\end{definition}

\begin{lemma}
Homotopy colimits commute with tensor products, pullbacks and
pushforwards. In particular, we have
\begin{enumerate}
\item $\hocolim \mathcal{H}^i(\mathcal{C}_e)=\mathcal{H}^i(\hocolim
\mathcal{C}_e)$, and
\item $\hocolim R^i\Gamma(\mathcal{C}_e)=R^i\Gamma(\hocolim
\mathcal{C}_e)$
\end{enumerate}
\end{lemma}

Similarly, the homotopy limit is defined as:

\begin{definition}
Let $\{\mathcal{C}_e\}$ be an inverse system of objects in
$D_{qc}(X)$,
$$\mathcal{C}_1\xleftarrow{f_1}\mathcal{C}_2\xleftarrow{f_2}\cdots.$$
The homotopy limit $\holim \mathcal{C}_e$ is defined by the
following triangle
$$\holim \mathcal{C}_e \to \prod \mathcal{C}_e\xrightarrow{\text{id}-\prod f_e}
\prod \mathcal{C}_e \to \holim \mathcal{C}_e[1].$$
\end{definition}

If $\mathcal{C}_e$ are coherent sheaves, then $\hocolim
\mathcal{C}_i=\varinjlim \mathcal{C}_e$.

\begin{lemma}
If $\{\mathcal{C}_e\}$ is a direct system in $D_{qc}(X)$ and
$\mathcal{D}\in D_{qc}(X)$, then
$$R\mathcal{H}om(\hocolim \mathcal{C}_e, \mathcal{D})\cong \holim
R\mathcal{H}om(\mathcal{C}_e,\mathcal{D}).$$ In particular,
$$D_X(\hocolim \mathcal{C}_e)\cong \holim D_X(\mathcal{C}_e).$$
\end{lemma}

\subsection{Fourier-Mukai transform}
Let $\hat{A}$ be the dual abelian variety of $A$. Let $P$ be the
normalized Poincare line bundle on $A\times \hat{A}$. Let $p_A$ and
$p_{\hat{A}}$ be the projection from $A\times \hat{A}$ to $A$ and
$\hat{A}$, respectively. Let $\hat{S}$ be the functor between
$\mathcal{O}_A$-modules and $\mathcal{O}_{\hat{A}}$-modules defined
as:
$$\hat{S}(\mathcal{F})=p_{\hat{A},*}(p_A^*\mathcal{F}\otimes P).$$
The \textbf{Fourier-Mukai transform} $R\hat{S}:D(A)\to D(\hat{A})$
is the right derived functor of $\hat{S}$. Similarly, we define
$RS:D(\hat{A})\to D(A)$ as the right derived functor of
$S(\mathcal{G})=p_{A,*}(p_{\hat{A}}^*\mathcal{G}\otimes P).$ We
recall the following propositions from \cite{M81} and \cite{HP}.

\begin{proposition}\label{FM_inverse} (See \cite[Theorem 2.2]{M81}\cite[Theorem
2.18]{HP}) The following properties hold on $D_{qc}(A)$ and
$D_{qc}(\hat{A})$.
$$RS \circ R\hat{S}=(-1_A)^*[-g] \qquad R\hat{S} \circ
RS=(-1_{\hat{A}})^*[-g],$$ where $-1_A$ is the inverse on $A$ and
$[-g]$ denotes the shift by $g$ places to the right.
\end{proposition}

\begin{proposition}\label{FM_vs_Ext} (See \cite[Corollary 2.5]{M81})
For all objects $\mathcal{E},\mathcal{E}'\in D_{qc}(A)$,
$$\Hom_{D_{qc}(A)}(\mathcal{E},\mathcal{E}')\cong
\Hom_{D_{qc}(\hat{A})}(R\hat{S}(\mathcal{E}),R\hat{S}(\mathcal{E}')).$$
\end{proposition}

\begin{proposition}\label{FM_vs_dual} (See
\cite[3.8]{M81}\cite[Lemma 2.20]{HP}) We have $D_A\circ
RS=(-1_A)^*(R\hat{S}\circ D_{\hat{A}})[g]$ on $D_{qc}(A)$.
\end{proposition}

The Fourier-Mukai transform commutes with homotopical colimit.

\begin{proposition}\label{FM_vs_colim} (See \cite[Lemma 2.23]{HP})
Let $\{\Lambda_e\}$ be a direct system in $D_{qc}(A)$, Then
$R\hat{S}(\hocolim \Lambda_e)=\hocolim R\hat{S}(\Lambda_e)$.
\end{proposition}

The Fourier-Mukai transform exchanges direct and inverse images of
isogenies.

\begin{proposition}\label{FM_vs_pullback} (See \cite[3.4]{M81}\cite[Lemma 2.22]{HP})
Let $\phi:A\to B$ be an isogeny of abelian varieties and
$\hat{\phi}:\hat{B}\to \hat{A}$ be the dual isogeny. The following
equalities hold on $D_{qc}(B)$ and $D_{qc}(A)$:
$$\phi^*\circ RS_B\cong RS_A\circ \hat{\phi}_*,$$
$$\phi_*\circ RS_A\cong RS_B\circ \hat{\phi}^*.$$
\end{proposition}

We will use the following consequence of the projection formula:

\begin{proposition}\label{FM_vs_projformula} (See \cite[Lemma
2.1]{PP11}) For all objects $\mathcal{E}\in D_c(A)$ and
$\mathcal{E}'\in D_c(\hat{A})$,
$$H^i(A, \mathcal{E}\otimes RS(\mathcal{E}'))=H^i(\hat{A},
R\hat{S}(\mathcal{E})\otimes \mathcal{E}').$$
\end{proposition}

\subsection{Inverse limit}

We refer to \cite[Chapter I, \S 4]{H75} and \cite[$0_{\text{III}}$,
\S 13]{EGA} for details in this section.

Let $\{\Omega_e\}$ be an inverse system of coherent sheaves. We say
$\{\Omega_e\}$ satisfies \textbf{the Mittag-Leffler condition}, if
for any $e\geq 0$ the image of $\Omega_{e'}\to \Omega_{e}$
stabilized for $e'$ sufficiently large. The inverse limit functor is
always left exact in the sense that if $\{\mathcal{F}_e\}$,
$\{\mathcal{G}_e\}$ and $\{\mathcal{H}_e\}$ are inverse systems of
coherent sheaves and the following exact sequences
$$0 \to \mathcal{F}_e \to \mathcal{G}_e \to \mathcal{H}_e \to 0$$
are compatible with maps in the inverse systems, then $$0 \to
\varprojlim\mathcal{F}_e \to \varprojlim\mathcal{G}_e \to
\varprojlim\mathcal{H}_e$$ is exact in the category of
quasi-coherent sheaves. By a theorem of Roos \cite{R61}, the right
derived functors $R^i\varprojlim=0$ for $i\geqslant 2$. Hence, we
have a long exact sequence
$$0 \to
\varprojlim\mathcal{F}_e \to \varprojlim\mathcal{G}_e \to
\varprojlim\mathcal{H}_e \to R^1\varprojlim\mathcal{F}_e \to
R^1\varprojlim\mathcal{G}_e \to R^1\varprojlim\mathcal{H}_e \to 0.$$

\begin{lemma} \label{ML_vs_projlim} (See \cite[Corollary I.4.3]{H75})
If $\{\Omega_e\}$ satisfies the Mittag-Leffler condition, then
$R^1\varprojlim\Omega_e=0$.
\end{lemma}

\begin{theorem} \label{commute_projlim} (See \cite[Theorem I.4.5]{H75})
Let $\{\Omega_e\}$ be an inverse system of coherent sheaves on a
variety $X$. Let $T$ be a functor on $D(X)$ which commutes with
arbitrary direct products. Suppose that $\{\Omega_e\}$ satisfies the
Mittag-Leffler condition. Then for each $i$, there is an exact
sequence $$0\to R^1\varprojlim R^{i-1}T(\Omega_e)\to
R^iT(\varprojlim \Omega_e)\to \varprojlim R^iT(\Omega_e) \to 0.$$ In
particular, if for some $i$, $\{R^{i-1}T(\Omega_e)\}$ satisfies the
Mittag-Leffler condition, then $R^iT(\varprojlim \Omega_e)\cong
\varprojlim R^iT(\Omega_e)$.
\end{theorem}

In applications, the functor $T$ above can be $\Gamma$, $f_*$, $S$
and $\hat{S}$ as in the Fourier-Mukai transform.

\subsection{Spectral sequence}

We recall the definition of spectral sequence from
\cite[0$_{\text{III}}$, \S 11]{EGA}. We also refer to
\cite[III.7]{GM03}. Let $\mathscr{C}$ be an abelian category. A
\textbf{(biregular) spectral sequence} $E$ on $\mathscr{C}$ consists
of the following ingredients:
\begin{enumerate}
\item A family of objects $\{E^{p,q}_{r}\}$ in $\mathscr{C}$,
where $p,q,r\in \mathbb{Z}$ and $r\geqslant 2$, such that for any
fixed pair $(p,q)$, $E^{p,q}_r$ stabilizes when $r$ is sufficiently
large. We denote the stable objects by $E^{p,q}_\infty$.
\item A family of morphisms $d^{p,q}_r:E^{p,q}_r\to
E^{p+r,q-r+1}_r$ satisfying $$d^{p+r,q-r+1}_r\circ d^{p,q}_r=0.$$
\item A family of isomorphisms
$\alpha^{p,q}_r:\ker(d^{p,q}_r)/\im(d^{p-r,q+r-1}_r)\xrightarrow{\sim}
E^{p,q}_{r+1}$.
\item A family of objects $\{E^n\}$ in $\mathscr{C}$. For every
$E_n$, there is a bounded decreasing filtration $\{F^pE^n\}$ in the
sense that there is some $p$ such that $F^pE^n=E^n$ and there is
some $p$ such that $F^pE^n=0$.
\item A family of isomorphisms
$\beta^{p,q}:E^{p,q}_\infty\xrightarrow{\sim}F^pE^{p+q}/F^{p+1}E^{p+q}$.
\end{enumerate}
We say the spectral sequence $\{E^{p,q}_r\}$ converges to $\{E^n\}$
and write $$E^{p,q}_2\Rightarrow E^{p+q}.$$

A morphism $\phi:E\to H$ between two spectral sequences on
$\mathscr{C}$ is a family of morphisms $\phi^{p,q}_r:E^{p,q}_r\to
H^{p,q}_r$ and $\phi^n:E^n\to H^n$ such that $\phi$ is compatible
with $d$, $\alpha$, the filtration and $\beta$.

\begin{theorem}
[Grothendieck] Let $\mathscr{A}, \mathscr{B}$ and $\mathscr{C}$ be
abelian categories. Let $F:\mathscr{A}\to \mathscr{B}$ and
$G:\mathscr{B}\to \mathscr{C}$ be two left exact functors. Suppose
every object in $\mathscr{A}$ and $\mathscr{B}$ has finite injective
resolution and the class of injective objects in $\mathscr{B}$ is
sufficiently large. Then for any object $X$ in $\mathscr{A}$, there
exists a spectral sequence with $E^{p,q}_2=R^pG(R^qF(X))$ converging
to $E^n=R^n(G\circ F)(X)$. It is functorial in $X$.
\end{theorem}

In this paper, we need to know when the morphisms between the limits
are zero.

\begin{lemma}\label{spec_seq_zero_map}
Let $$\xymatrix{ E_2^{i,j} \ar@2{->}[r] \ar^{\phi_2^{i,j}}[d] & E^{i+j} \ar^{\phi^{i+j}}[d] \\
H_2^{i,j} \ar@2{->}[r] & H^{i+j}}$$ be two spectral sequences with
commutative maps. Let $l$ and $a$ be integers. Suppose that
$E_2^{i,l-i}=0$ for $i<a$, $H_2^{i,l-i}=0$ for $i>a$ and
$\phi_2^{a,l-a}=0$. Then $\phi^l=0$.
\end{lemma}

\begin{proof}
Since $$E_3^{i,j}\cong \ker(E_2^{i,j}\to
E_2^{i+2,j-1})/\im(E_2^{i-2,j+1}\to E_2^{i,j}),$$ it follows that
$E_3^{i,l-i}=0$ for $i<a$, $H_3^{i,l-i}=0$ for $i>a$ and
$\phi_3^{a,l-a}=0$. Hence $E_\infty^{i,l-i}=0$ for $i<a$,
$H_\infty^{i,l-i}=0$ for $i>a$ and $\phi_\infty^{a,l-a}=0$, by
induction. Let $\{F^pE^l\}$ and $\{F^pH^l\}$ be the filtration for
$E^l$ and $H^l$, respectively.

We prove by induction that $F^p\phi^l:F^pE^l\to F^pH^l$ is zero for
any $p$. For $p\geq a+1$, we have that $F^pH^l=0$. Hence
$F^p\phi^l=0$. Suppose $F^{i+1}\phi^l=0$ for some $i\leq a$. Since
we have the following commutative diagram with exact rows:
$$\xymatrix{
0 \ar[r]& F^{i+1}E^l \ar[r]\ar^{F^{i+1}\phi^l}[d]& F^iE^l
\ar[r]\ar^{F^i\phi^l}[d]&
E_\infty^{i,l-i} \ar[r]\ar^{\phi_\infty^{i,l-i}}[d]& 0 \\
0 \ar[r]& F^{i+1}H^l \ar[r]& F^iH^l \ar[r]& H_\infty^{i,l-i} \ar[r]&
0 }$$

\begin{enumerate}
\item If $i=a$, then $F^{i+1}H^l=0$. By the Snake Lemma, we have that
$$\coker F^i\phi^l\cong \coker \phi_\infty^{i,l-i}\cong
H^{i,l-i}\cong F^iH^l.$$ Hence, $F^i\phi^l=0$.

\item If $i<a$, then
$E_\infty^{i,l-i}=0$. By induction, we may assume that
$F^{i+1}\phi^l=0$. By the Snake Lemma, we have that
$$\ker F^i\phi^l\cong \ker F^{i+1}\phi^l \cong F^{i+1}E^l \cong
F^iE^l.$$ We also obtain that $F^i\phi^l=0$.
\end{enumerate}

The lemma follows.
\end{proof}

\subsection{Frobenius morphism and Cartier module}

Let $k$ be a perfect field of positive characteristic. Let $X$ be a
normal variety over $k$. The (absolute) Frobenius morphism $F:X\to
X$ is defined as identity on the topological space and taking $p$-th
power on local sections. We denote by $F^e$ the $e$-th iteration of
$F$.

A \textbf{Cartier module} is a coherent sheaf $\mathcal{F}$ on $X$
equipped with an $\mathcal{O}_X$-linear map
$$\phi:F^e_*\mathcal{F}\to \mathcal{F},$$ which is also called a
$p^{-e}$-linear map in \cite{BS}. A well-known example of Cartier
module is the canonical sheaf $\omega_X$ with the trace map $$\Tr:
F_*\omega_X\to \omega_X,$$ which is the dual of the structure map
$\mathcal{O}_X\to F_*\mathcal{O}_X$.

Suppose $(\mathcal{F},\phi)$ is a Cartier module on $X$. It is easy
to see that we can iterate $\phi$ to get a sequence of maps
$$\cdots \to F^{3e}_*\mathcal{F} \xrightarrow{F^{2e}_*\phi}
F^{2e}_*\mathcal{F} \xrightarrow{F^e_*\phi} F^e_*\mathcal{F}
\xrightarrow{\phi} \mathcal{F}.$$ It is known that this inverse
system of coherent sheafs satisfies the Mittag-Leffler condition
\cite[Proposition 8.1.4]{BS}.

\section{Examples}

We will propose several examples for the pathologies appearing in
the context of Cartier modules.

\begin{example} \label{exmpl_HP}
This example first appears in \cite[Example 3.21]{HP}. Let $A$ be an
elliptic curve, $\Omega_0=\omega_A$, $\Omega_e=F^e_*\Omega_0$ and
$\alpha:F_*\Omega_0\to \Omega_0$ be the trace map.

When $A$ in ordinary, then $\Lambda=\hocolim
R\hat{S}(D_A(\Omega_e))=\bigoplus_{y\in A_p}k(y)$ where $A_p$
denotes the set of all $p^\infty$-torsion points in $\hat{A}$ and
$R\hat{S}(\Omega)=\prod_{y\in A_p}k(y)[-1]$. Hence
$$\Supp R^1\hat{S}(\Omega)=A_p,$$ which is a countable dense set by
\cite[5.30]{MvdG}. By \cite[Proposition 3.18]{HP}, $V^1(\Omega)=A_p$
which is dense in $\hat{A}$.

However, suppose $\mathcal{F}$ is a coherent sheaf satisfying the
Generic Vanishing conditions, then $V^1(\mathcal{F})$ is a closed
subset of dimension 0 or empty by Theorem \ref{main_PP11}.

We should notice that in this example, although $\Supp
R^1\hat{S}(\Omega)$ is not closed, the support of the image of
$R^1\hat{S}(\Omega)\to R^1\hat{S}(\Omega_e)$ is the set of
$p^e$-torsion points which is closed for any $e>0$. \qed
\end{example}

It should be noticed that in the previous example,
$\{R^i\hat{S}(\Omega_e)\}$ satisfies the Mittag-Leffler condition
for any $i$. We will see that this is not valid in general in the
following examples.

\begin{example} \label{notML_notCar}
Let $A$ be an elliptic curve. Let $\hat{0}\in\hat{A}$ correspond to
the trivial line bundle on $A$. Let
$W_e=\mathcal{O}_{\hat{A}}(-e\cdot \hat{0})$ and $\psi_e:W_{e+1}\to
W_e$ be the inclusion. Clearly, the inverse system of coherent
sheaves $\{W_e\}$ does not satisfy the Mittag-Leffler condition.
Since the $W_e$ are antiample, $R^0S(W_e)=0$. Let
$\Omega_e=R^1S(W_e)=RS(W_e)[1]$. Notice that we have short exact
sequences $$0\to W_{e+1} \to W_e \to k(\hat{0})\to 0.$$ The
Fourier-Mukai transform induces the following long exact sequences
$$\xymatrix{
R^0S(W_e) \ar[r]\ar@{=}[d] & R^0S(k(\hat{0})) \ar[r]\ar@{=}[d] &
R^1S(W_{e+1}) \ar[r]\ar@{=}[d] & R^1S(W_e) \ar[r]\ar@{=}[d] &
R^1S(k(\hat{0})) \ar@{=}[d]\\
0 & \mathcal{O}_A & \Omega_{e+1} & \Omega_e & 0. }$$ In particular,
$\Omega_{e+1}\to \Omega_e$ is surjective. Thus $\{\Omega_e\}$
satisfies the Mittag-Leffler condition. On the other hand,
$R^0\hat{S}(\Omega_e)=(-1_{\hat{A}})^*W_e$ does not satisfy the
Mittag-Leffler condition.

We claim that $\{\Omega_e\}$ above is a GV-inverse system of
coherent sheaves. Indeed,
$$R\hat{S}(D_A(\Omega_e))=(-1_{\hat{A}})^*D_{\hat{A}}(R\hat{S}(\Omega_e))[-1]=
D_{\hat{A}}(W_e)[-1]=\mathcal{O}_{\hat{A}}(e\cdot \hat{0}).$$ In
particular, $R^{-1}\hat{S}(D_A(\Omega_e))=0$. Taking the direct
limit, we conclude that $\mathcal{H}^{-1}(\Lambda)=0$.

We first calculate the codimensions of the supports. Since
$R^1\hat{S}(\Omega_e)=0$ for any $e$, it is clear that $\varprojlim
R^1\hat{S}(\Omega_e)=0$. Since $\{\Omega_e\}$ satisfies the
Mittag-Leffler condition, by Theorem \ref{commute_projlim}, we have
the following exact sequence
$$0\to R^1\varprojlim R^0\hat{S}(\Omega_e) \to R^1\hat{S}(\Omega)\to
\varprojlim R^1\hat{S}(\Omega_e)\to 0.$$ Thus $R^1\hat{S}(\Omega)
\cong R^1\varprojlim R^0\hat{S}(\Omega_e)\cong
(-1_{\hat{A}})^*R^1\varprojlim W_e$. Consider the following exact
sequence of inverse systems
$$0\to W_e \to \mathcal{O}_{\hat{A}} \to k[t]/(t^e) \to 0.$$ We have long
exact sequence $$0\to \varprojlim W_e \to \mathcal{O}_{\hat{A}} \to
\varprojlim k[t]/(t^e) \to R^1\varprojlim W_e \to R^1\varprojlim
\mathcal{O}_{\hat{A}}=0,$$ where the last equation follows by Lemma
\ref{ML_vs_projlim}. We conclude that $R^1\varprojlim W_e$ is the
skyscraper sheaf $k[[t]]/k[t]$ at $\hat{0}$. Hence
$$\Supp R^1\hat{S}(\Omega)=\{\hat{0}\}$$ and $$\Supp (\im
R^1\hat{S}(\Omega)\to R^1\hat{S}(\Omega_e))=\emptyset.$$ By Theorem
\ref{commute_projlim}, it is clear that
$R^0\hat{S}(\Omega)=\varprojlim
R^0\hat{S}(\Omega_e)=(-1)^*\varprojlim W_e$. Hence $$\Supp
R^0\hat{S}(\Omega)=\hat{A}-\{\hat{0}\}$$ and $$\Supp (\im
R^0\hat{S}(\Omega)\to R^0\hat{S}(\Omega_e))=\hat{A}-\{\hat{0}\}.$$

We now calculate the cohomology support loci. Let $\alpha\in
\hat{A}$ and $P_\alpha\in \Pic^0(A)$ be the corresponding
topologically trivial line bundle. We have
$$\begin{array}{r@{\;\;\;\cong\;\;\;}l}
H^i(A, \Omega_e\otimes P_\alpha) & H^i(A, R^1S(W_e)\otimes P_\alpha)
\\
& H^{i+1}(A, RS(W_e)\otimes P_\alpha) \\
& H^{i+1}(\hat{A}, W_e\otimes R\hat{S}(P_\alpha)) \\
& H^i(\hat{A}, W_e\otimes R^1\hat{S}(P_\alpha)) \\
& H^i(\hat{A}, W_e\otimes k(-\alpha)),
\end{array}$$
where the third isomorphism is by Proposition
\ref{FM_vs_projformula}. Hence $H^0(A,\Omega_e\otimes P_\alpha)\cong
k$ and $H^1(A,\Omega_e\otimes P_\alpha)=0$. Taking the inverse
limit, we have $H^i(A,\Omega\otimes P_\alpha)=\varprojlim
H^i(A,\Omega_e\otimes P_\alpha)=k(-\alpha)$ if $i=0$ and $\alpha\neq
\hat{0}$, and $H^i(A,\Omega\otimes P_\alpha)=0$ otherwise. We
conclude that $$V^1(\Omega)=\emptyset$$ and
$$V^0(\Omega)=\hat{A}-\{\hat{0}\}.$$
In particular, $$V^0(\Omega)=\Supp R^0\hat{S}(\Omega)=\Supp (\im
R^0\hat{S}(\Omega)\to R^0\hat{S}(\Omega_e))=\hat{A}-\{\hat{0}\}$$
are not countable unions of closed subsets.

We should notice that $V^1(\Omega)\nsupseteq \Supp
R^1\hat{S}(\Omega)$. If $\mathcal{F}$ is a coherent sheaf, then it
is a consequence of cohomology and base change that
$V^i(\mathcal{F})\supseteq \Supp R^i\hat{S}(\mathcal{F})$. \qed

\end{example}

We can easily modify Example \ref{notML_notCar} to obtain a Cartier
module.

\begin{example} \label{notML_Car}
Let $A$ be a supersingular elliptic curve. Let $\{W_e\}$ and
$\{\Omega_e\}$ be the same as in Example \ref{notML_notCar}. Let us
consider the inverse system $\{\Omega_{p^e}\}$. Notice that since
$A$ is supersingular,
$$W_{p^e}=\mathcal{O}_{\hat{A}}(-p^e\cdot \hat{0})=V^{*,e}(W_1),$$
where the Verschiebung $V:\hat{A}\to\hat{A}$ is the dual of the
Frobenius. We have
$$\Omega_{p^e}\cong RS(W_{p^e})[1]\cong RS(V^{*,e}(W_1))[1]\cong F_*^e(RS(W_1))[1]\cong F_*^e\Omega_{1},$$
where the third isomorphism is by Proposition \ref{FM_vs_pullback}.
Thus, $\{\Omega_{p^e}\}$ is a Cartier module. The calculation of
inverse limits remains unchanged as in Example \ref{notML_notCar}.
In particular, $V^0(\Omega)=\hat{A}-\{\hat{0}\}$ is not a countable
union of closed subvarieties. This gives a negative answer to
\cite[Question 3.20]{HP}. \qed
\end{example}

The following example shows that the chain of inclusions fails for
GV-inverse system of coherent sheaves.

\begin{example}\label{chain_failure}

Let $A$ be an elliptic curve. Let $\Omega_0=\mathcal{O}_A$ and
$\Omega_{e+1}$ be the non-splitting extension of $\mathcal{O}_A$ and
$\Omega_e$,
$$0\to \mathcal{O}_A \to \Omega_{e+1}\to \Omega_e\to 0.$$
Then $H^0(A,\Omega_e)\cong H^1(A,\Omega_e)\cong k$ for any
$e\geqslant 0$. Since $\Omega_{e+1}\to\Omega_e$ is surjective, the
inverse system $\{\Omega_e\}$ satisfies the Mittag-Leffler
condition. It is easy to check that
$$R^{-1}\hat{S}(D_A(\Omega_e))=0$$ for any $e\geqslant 0$ by induction. Hence,
$\{\Omega_e\}$ is a GV-inverse system.

We now compute the cohomology support loci $V^0(\Omega)$ and
$V^1(\Omega)$. Suppose $P_\alpha\in \Pic^0(A)$ and $P_\alpha\neq
\mathcal{O}_A$. By the long exact sequence,
\begin{eqnarray*}
\lefteqn{0 \to H^0(A, P_\alpha) \to H^0(A, \Omega_{e+1}\otimes
P_\alpha) \to H^0(A, \Omega_e\otimes
P_\alpha)}\\
& & \to H^1(A, P_\alpha) \to H^1(A, \Omega_{e+1}\otimes P_\alpha)
\to H^1(A, \Omega_e\otimes P_\alpha) \to 0,\end{eqnarray*} we have
that $$H^0(A, \Omega_e\otimes P_\alpha)=H^1(A, \Omega_e\otimes
P_\alpha)=0$$ for any $e\geqslant 0$. Thus, we only need to consider
whether $\hat{0}\in\hat{A}$ is in the cohomology support loci of
$\Omega$.

We have the following exact sequence,
\begin{eqnarray*} \lefteqn{0
\to H^0(A, \mathcal{O}_A) \to H^0(A, \Omega_{e+1}) \to H^0(A,
\Omega_e)}\\
& & \xrightarrow{\simeq} H^1(A, \mathcal{O}_A) \to H^1(A,
\Omega_{e+1})\to H^1(A, \Omega_e) \to 0,\end{eqnarray*} where the
isomorphism is by the assumption that the extension of
$\mathcal{O}_A$ and $\Omega_e$ is non-splitting. Hence, $H^0(A,
\Omega_{e+1}) \to H^0(A, \Omega_e)$ is zero and $H^1(A,
\Omega_{e+1}) \to H^1(A, \Omega_e)$ is an isomorphism for any
$e\geqslant 0$. Taking the inverse limit, we obtain that
$$H^0(A, \Omega)=H^0(A, \varprojlim \Omega_e)\cong \varprojlim
H^0(A,\Omega_e)=0,$$ and
$$H^1(A, \Omega)=H^1(A, \varprojlim \Omega_e)\cong \varprojlim
H^1(A,\Omega_e)\cong H^1(A,\Omega_0)\cong k.$$ Thus,
$V^0(\Omega)=\emptyset$ and $V^1(\Omega)=\{\hat{0}\}$. The chain of
inclusions fails.

When $A$ is supersingular, by \cite[Lemma 4.12]{HST}, we have that
$F^e_*\omega_A\cong \Omega_{p^e-1}$. The nontrivial map
$\Omega_{p^e-1}\to \Omega_0$ induces $F^e_*\omega_A\to \omega_A$,
which is isomorphic to the trace map up to a scale. Hence, the
inverse system $\{\Omega_{p^e-1}\}$ is a Cartier module and is the
same as Example \ref{exmpl_HP}. \qed

\end{example}

\section{Main Theorem}

We will prove Theorem \ref{main} in this section.

\subsection{WIT versus limit of Kodaira vanishing}

\begin{theorem}\label{WIT=limKV}
Let $A$ be an abelian variety of dimension $g$. Let $\{\Omega_e\}$
be an inverse system of coherent sheaves on $A$ satisfying the
Mittag-Leffler condition and let $\Omega=\varprojlim \Omega_e$. Let
$\Lambda_e=R\hat{S}(D_A(\Omega_e))$ and $\Lambda=\hocolim\Lambda_e$.
The following are equivalent:
\begin{enumerate}
\item For any ample line bundle $L$ on $\hat{A}$,
$H^i(A,\Omega\otimes \hat{L}^\vee)=0$ for any $i>0$.
\item For any non-negative integer $e$ and any ample line bundle
$L$ on $\hat{A}$, there exists an integer $m(e,L)$ such that for any
$m\geq m(e,L)$, the natural map $$H^i(A,\Omega\otimes
\widehat{mL}^\vee)\to H^i(A,\Omega_e\otimes \widehat{mL}^\vee)$$ is
zero for any $i>0$.
\item $\mathcal{H}^i(\Lambda)=0$ for $i\neq 0$.
\end{enumerate}
\end{theorem}

\begin{proof}
$(1)\Rightarrow (2)$. Obvious.

$(2)\Rightarrow (3)$. It is shown in \cite[Theorem 3.1.1]{HP} that
$\mathcal{H}^i(\Lambda)=0$ when $i<-g$ or $i>0$. Thus we may pick
$j<0$ as the smallest integer such that $\mathcal{H}^j(\Lambda)\neq
0$. Since $\mathcal{H}^j(\Lambda)=\hocolim\mathcal{H}^j(\Lambda_e)$,
we may fix $e>0$ such that the image of $\mathcal{H}^j(\Lambda_e)\to
\mathcal{H}^j(\Lambda)$ is non-zero. Let $L$ be a sufficiently large
multiple of a fixed ample line bundle on $\hat{A}$ such that

{\renewcommand{\labelenumi}{(\roman{enumi})}

\begin{enumerate}
\item $\mathcal{H}^j(\Lambda_e)\otimes L$ is globally generated,
\item $H^i(\hat{A},\mathcal{H}^l(\Lambda_e)\otimes L)=0$ for $i>0$
and $l\in[-g,0]$, and
\item $H^i(A,\Omega\otimes \hat{L}^\vee)\to H^i(A,\Omega_e\otimes
\hat{L}^\vee)$ is zero for any $i\neq 0$.
\end{enumerate}}

Notice that (i) and (ii) can be achieved by Serre Vanishing and
(iii) can be achieved by the hypothesis in condition (2).

Using Grothendieck's spectral sequence, we have
$$\xymatrix{
E_{2,e}^{i,l}=R^i\Gamma(\mathcal{H}^l(\Lambda_e)\otimes L)
\ar@2{->}[r] \ar[d] & R^{i+l}\Gamma(\Lambda_e\otimes L) \ar[d] \\
E_2^{i,l}=R^i\Gamma(\mathcal{H}^l(\Lambda)\otimes L) \ar@2{->}[r] &
R^{i+l}\Gamma(\Lambda\otimes L), }$$ where the vertical arrows are
compatible by the functoriality of the spectral sequence. By our
choice of $j$, we have that $E_2^{i,l}=0$ for all $l<j$. By (ii),
$E_{2,e}^{i,l}=0$ for any $i\neq 0$. Hence, the spectral sequence
degenerates to the following commutative diagram:
$$\xymatrix{
R^0\Gamma(\mathcal{H}^j(\Lambda_e)\otimes L)
\ar^-{\simeq}[r] \ar[d] & R^j\Gamma(\Lambda_e\otimes L) \ar[d] \\
R^0\Gamma(\mathcal{H}^j(\Lambda)\otimes L) \ar^-{\simeq}[r] &
R^j\Gamma(\Lambda\otimes L). }$$ By (i), the image of
$R^0\Gamma(\mathcal{H}^j(\Lambda_e)\otimes L)\to
R^0\Gamma(\mathcal{H}^j(\Lambda)\otimes L)$ is non-zero. Hence, the
image of $R^j\Gamma(\Lambda_e\otimes L)\to R^j\Gamma(\Lambda\otimes
L)$ is non-zero.

On the other hand, we have
$$\begin{array}{r@{\;\;\;\cong\;\;\;}l}
D_k(R^j\Gamma(\Lambda\otimes L)) & D_k(\varinjlim
R^j\Gamma(\Lambda_e\otimes L)) \\
& \varprojlim D_kR^j\Gamma(\Lambda_e\otimes L) \\
& \varprojlim D_kR^j\Gamma(R\hat{S}(D_A(\Omega_e))\otimes L) \\
& \varprojlim D_kR^j\Gamma(D_A(\Omega_e\otimes \hat{L}^\vee)) \\
& \varprojlim R^{-j}\Gamma(D_A(D_A(\Omega_e\otimes \hat{L}^\vee)))
\\
& \varprojlim R^{-j}\Gamma(\Omega_e\otimes \hat{L}^\vee),
\end{array}$$ and similarly,
$D_k(R^j\Gamma(\Lambda_e\otimes L))\cong
R^{-j}\Gamma(\Omega_e\otimes \hat{L}^\vee)$. Since the inverse
system $\{\Omega_e\}$ satisfies the Mittag-Leffler condition, we
have
$$H^i(A,\Omega\otimes \hat{L}^\vee)=H^i(A,\varprojlim
\Omega_e\otimes \hat{L}^\vee)\cong\varprojlim H^i(A,\Omega_e\otimes
\hat{L}^\vee)$$ for any $i$. Thus by (iii), $\varprojlim
R^{-j}\Gamma(\Omega_e\otimes \hat{L}^\vee)\to
R^{-j}\Gamma(\Omega_e\otimes \hat{L}^\vee)$ is zero. Hence,
$D_k(R^j\Gamma(\Lambda\otimes L)) \to D_k(R^j\Gamma(\Lambda_e\otimes
L))$ is zero, a contradiction.

$(3)\Rightarrow (1)$. Recall that we have the following spectral
sequence,
$$H^i(\hat{A},\mathcal{H}^l(\Lambda)\otimes L) \Rightarrow
R^{i+l}\Gamma(\hat{A},\Lambda\otimes L). $$ Since
$\mathcal{H}^i(\Lambda)=0$ for any $i\neq 0$, the spectral sequence
degenerates to $$H^i(\hat{A},\mathcal{H}^0(\Lambda)\otimes L)\cong
R^i\Gamma(\hat{A}, \Lambda\otimes L).$$ If $i>0$, then by the
isomorphism shown in the previous step,
$$H^i(A, \Omega\otimes \hat{L}^\vee)\cong D_k(R^{-i}\Gamma(\hat{A},
\Lambda\otimes L))\cong
D_k(H^{-i}(\hat{A},\mathcal{H}^0(\Lambda)\otimes L))=0.$$
\end{proof}

\subsection{WIT versus the supports of $R^i\hat{S}(\Omega)$}

\begin{theorem} \label{WITtoGV}
Let $\{\Omega_e\}$ be an inverse system of coherent sheaves on a
$g$-dimensional abelian variety satisfying the Mittag-Leffler
condition and let $\Omega=\varprojlim \Omega_e$. Let
$\Lambda_e=R\hat{S}(D_A(\Omega_e))$ and $\Lambda=\hocolim\Lambda_e$.
If $\mathcal{H}^j(\Lambda)=0$ for any $j\neq 0$, then for any
scheme-theoretic point $P$ with $\dim P>i$, we have $$P\notin \Supp
(\im(R^i\hat{S}(\Omega)\to R^i\hat{S}(\Omega_e)))$$ for any $e\geq
0$. Moreover, if $\mathcal{H}^0(\Lambda)$ is torsion-free, then for
any scheme-theoretic point $P$ with $\dim P\geq i$, we have
$$P\notin \Supp (\im(R^i\hat{S}(\Omega)\to R^i\hat{S}(\Omega_e)))$$
for any $e\geq 0$.
\end{theorem}

\begin{proof}
Fix a scheme-theoretic point $P\in \hat{A}$ such that $\dim P=d$.
Since localization at $P$ is exact, we have the following
commutative diagram of spectral sequences:

$$\xymatrix{
\mathcal{E}xt^i(\mathcal{H}^j(\Lambda),\mathcal{O}_{\hat{A}})_P
\ar@2{->}[r]\ar^{\phi_2^{i,-j}}[d] &
\mathcal{E}xt^{i-j}(\Lambda,\mathcal{O}_{\hat{A}})_P \ar^{\phi^{i-j}}[d] \\
\mathcal{E}xt^i(\mathcal{H}^j(\Lambda_e),\mathcal{O}_{\hat{A}})_P
\ar@2{->}[r] & \mathcal{E}xt^
{i-j}(\Lambda_e,\mathcal{O}_{\hat{A}})_P. }$$

Since $\dim P=d$, by \cite[III.6.8 and III.6.10A]{H77},
$$\mathcal{E}xt^i(\mathcal{H}^j(\Lambda_e),\mathcal{O}_{\hat{A}})_P\cong
Ext^i_{\mathcal{O}_{\hat{A},P}}(\mathcal{H}^j(\Lambda_e)_P,\mathcal{O}_{\hat{A},P})=0,$$
when $i>g-d$. Let $l=i-j>g-d$ and $a=l-1$. When $i\leq a$, we have
$j=i-l<0$, hence
$\mathcal{E}xt^i(\mathcal{H}^j(\Lambda),\mathcal{O}_{\hat{A}})_P=0$.
When $i>a$, we have $i\geq l>g-d$, hence
$\mathcal{E}xt^i(\mathcal{H}^j(\Lambda_e),\mathcal{O}_{\hat{A}})_P=0$.
We may apply Lemma \ref{spec_seq_zero_map} and obtain that the
natural map $\mathcal{E}xt^l(\Lambda,\mathcal{O}_{\hat{A}})_P\to
\mathcal{E}xt^l(\Lambda_e,\mathcal{O}_{\hat{A}})_P$ is zero when
$l>g-d$.

It is easy to see that
$$\begin{array}{cl}
 & \mathcal{E}xt^l(\Lambda_e,\mathcal{O}_{\hat{A}}) \cong
\mathcal{H}^{l-g}(D_{\hat{A}}(\Lambda_e)) \cong
\mathcal{H}^{l-g}(D_{\hat{A}}(R\hat{S}(D_A(\Omega_e)))) \\
\!\!\!\cong\!\!\! & \mathcal{H}^{l-g}((-1_{\hat{A}})^*
R\hat{S}(D_A(D_A(\Omega_e)))[g]) \cong
\mathcal{H}^l((-1_{\hat{A}})^* R\hat{S}(\Omega_e)) \\
\!\!\!\cong\!\!\! &
(-1_{\hat{A}})^*R^l\hat{S}(\Omega_e)\end{array}$$ and
$$\begin{array}{cl}
 & \mathcal{E}xt^l(\Lambda,\mathcal{O}_{\hat{A}}) \cong
\mathcal{H}^{l-g}(D_{\hat{A}}(\Lambda)) \cong
\mathcal{H}^{l-g}(D_{\hat{A}}(\hocolim R\hat{S}(D_A(\Omega_e))))
\\ \!\!\!\cong\!\!\! & \mathcal{H}^{l-g}(\holim
D_{\hat{A}}(R\hat{S}(D_A(\Omega_e)))) \cong \mathcal{H}^l(\holim
(-1_{\hat{A}})^*R\hat{S}(\Omega_e)) \\
\!\!\!\cong\!\!\! & (-1_{\hat{A}})^*\mathcal{H}^l(\holim
R\hat{S}(\Omega_e)).
\end{array}$$ Then for any $d$-dimensional point $P$, $$\mathcal{H}^l(\holim
R\hat{S}(\Omega_e))_P \to R^l\hat{S}(\Omega_e)_P$$ is zero for any
$l>g-d$. Notice that $R^l\hat{S}(\Omega)_P\to
R^l\hat{S}(\Omega_e)_P$ factors as
$$R^l\hat{S}(\Omega)_P = \mathcal{H}^l(R\hat{S}(\varprojlim
\Omega_e))_P\to \mathcal{H}^l(\holim R\hat{S}(\Omega_e))_P \to
\mathcal{H}^l(R\hat{S}(\Omega_e))_P = R^l\hat{S}(\Omega_e)_P.$$ We
conclude that $R^l\hat{S}(\Omega)_P\to R^l\hat{S}(\Omega_e)_P$ is
zero when $l>g-\dim P$. The first part of the proposition follows
from the exactness of localization at $P$.

If $\mathcal{H}^0(\Lambda)$ is torsion-free, we only need to check
the case that $l=i-j=g-d$. Let $a=l$. When $i<a$, we have that
$j=i-l<0$, hence
$\mathcal{E}xt^i(\mathcal{H}^j(\Lambda),\mathcal{O}_{\hat{A}})_P=0$.
When $i>a$, we have $i>l=g-d$, hence
$\mathcal{E}xt^i(\mathcal{H}^j(\Lambda_e),\mathcal{O}_{\hat{A}})_P=0$.
To apply Lemma \ref{spec_seq_zero_map}, we only need to check that
$$\mathcal{E}xt^{g-d}(\mathcal{H}^0(\Lambda),\mathcal{O}_{\hat{A}})_P\to
\mathcal{E}xt^{g-d}(\mathcal{H}^0(\Lambda_e),\mathcal{O}_{\hat{A}})_P$$
is zero. Let $\mathcal{T}$ be the torsion part of
$\mathcal{H}^0(\Lambda_e)$ and $\mathcal{F}\cong
\mathcal{H}^0(\Lambda_e)/\mathcal{T}$. We have the following exact
sequence,
$$\mathcal{E}xt^{g-d}(\mathcal{F},\mathcal{O}_{\hat{A}})_P\to
\mathcal{E}xt^{g-d}(\mathcal{H}^0(\Lambda_e),\mathcal{O}_{\hat{A}})_P\to
\mathcal{E}xt^{g-d}(\mathcal{T},\mathcal{O}_{\hat{A}})_P\to
\mathcal{E}xt^{g-d+1}(\mathcal{F},\mathcal{O}_{\hat{A}})_P.$$ Since
$\mathcal{F}$ is torsion-free, by the argument in \cite[Lemma
2.9]{PPIII}, we have that
$$\mathcal{E}xt^{g-d}(\mathcal{F},\mathcal{O}_{\hat{A}})_P=
\mathcal{E}xt^{g-d+1}(\mathcal{F},\mathcal{O}_{\hat{A}})_P=0.$$
Hence
$$\mathcal{E}xt^{g-d}(\mathcal{H}^0(\Lambda_e),\mathcal{O}_{\hat{A}})_P\cong
\mathcal{E}xt^{g-d}(\mathcal{T},\mathcal{O}_{\hat{A}})_P.$$ We only
need to show that
$$\mathcal{E}xt^{g-d}(\mathcal{H}^0(\Lambda),\mathcal{O}_{\hat{A}})_P\to
\mathcal{E}xt^{g-d}(\mathcal{T},\mathcal{O}_{\hat{A}})_P$$ is zero.
Notice that this map is induced by $\mathcal{T}\to
\mathcal{H}^0(\Lambda_e)\to \mathcal{H}^0(\Lambda)$ where
$\mathcal{T}$ is a torsion sheaf and $\mathcal{H}^0(\Lambda)$ is
torsion-free. Thus $\mathcal{T}\to \mathcal{H}^0(\Lambda)$ is zero.
Hence,
$\mathcal{E}xt^{g-d}(\mathcal{H}^0(\Lambda),\mathcal{O}_{\hat{A}})_P\to
\mathcal{E}xt^{g-d}(\mathcal{T},\mathcal{O}_{\hat{A}})_P$ is zero.
\end{proof}

\subsection{The case when $\{R^i\hat{S}(\Omega_e)\}$ satisfies the Mittag-Leffler condition}

In this section, we will consider the Mittag-Leffler condition on
the Fourier-Mukai transform of the inverse system $\{\Omega_e\}$. We
are able to recover Theorem \ref{main_PP11} fully in this setting.
However, we remind the reader that even when $\{\Omega_e\}$ is a
Cartier module, the inverse system $\{R^i\hat{S}(\Omega_e)\}$ does
not necessarily satisfy the Mittag-Leffler condition (see Example
\ref{notML_Car}).

\begin{proposition} \label{closed}
For any $0\leqslant i\leqslant g$, if $\{R^i\hat{S}(\Omega_e)\}$
satisfies the Mittag-Leffler condition, then the support of $\im
(R^i\hat{S}(\Omega)\to R^i\hat{S}(\Omega_e))$ is closed for any
$e\geqslant 0$.
\end{proposition}

\begin{proof}
Since $\{\Omega_e\}$ satisfies the Mittag-Leffler condition, by
Theorem \ref{commute_projlim}, the natural map
$$R^i\hat{S}(\Omega)\to \varprojlim R^i\hat{S}(\Omega_e)$$ is
surjective. Thus, we have
$$\im(R^i\hat{S}(\Omega)\to R^i\hat{S}(\Omega_e))=
\im(\varprojlim R^i\hat{S}(\Omega_e)\to R^i\hat{S}(\Omega_e)).$$
Since we assume $\{R^i\hat{S}(\Omega_e)\}$ satisfies the
Mittag-Leffler condition, the image of $$R^i\hat{S}(\Omega_d)\to
R^i\hat{S}(\Omega_e)$$ stabilizes when $d$ is sufficiently large.
The stable image coincides with $$\im(\varprojlim
R^i\hat{S}(\Omega_e)\to R^i\hat{S}(\Omega_e)).$$ Since
$R^i\hat{S}(\Omega_d)$ and $R^i\hat{S}(\Omega_e)$ are both coherent,
the proposition follows.
\end{proof}

By the proposition above, we are able to talk about the codimension
of the support of $\im (R^i\hat{S}(\Omega)\to
R^i\hat{S}(\Omega_e))$.




We are able to recover the missing implication in Theorem
\ref{main}.

\begin{theorem} \label{GVtolimKV}
Let $A$ be an abelian variety. Let $\{\Omega_e\}$ be an inverse
system of coherent sheaves on $A$ satisfying the Mittag-Leffler
condition and let $\Omega=\varprojlim \Omega_e$. Let
$\Lambda_e=R\hat{S}(D_A(\Omega_e))$ and $\Lambda=\hocolim\Lambda_e$.
Suppose that the inverse system $\{R^i\hat{S}(\Omega_e)\}$ satisfies
the Mittag-Leffler condition for all $i$ and
$$\codim \Supp (\im(R^i\hat{S}(\Omega)\to
R^i\hat{S}(\Omega_e)))\geq i,$$ for any $0\leqslant i\leqslant g$
and $e$ sufficiently large. Then for any ample line bundle $L$ on
$\hat{A}$, we have $H^i(A,\Omega\otimes \hat{L}^\vee)=0$ for any
$i>0$.
\end{theorem}

\begin{proof}
To simplify our notation, we denote
$$\im_{i,e}=\im(R^i\hat{S}(\Omega)\to R^i\hat{S}(\Omega_e)).$$
Let $p$ and $q$ be two non-negative integers satisfying $p+q>g$ and
$L$ be any ample line bundle on $\hat{A}$. The homomorphism
$$H^p(\hat{A},R^q\hat{S}(\Omega)\otimes L^\vee)\to H^p(\hat{A},R^q\hat{S}(\Omega_e)\otimes
L^\vee)$$ factors through $H^p(\hat{A},\im_{q,e}\otimes L^\vee)$.
Since we assume that $$\codim \Supp (\im_{q,e})\geqslant q>g-p,$$
the cohomology $$H^p(\hat{A},\im_{q,e}\otimes L^\vee)=0.$$ Thus,
$$H^p(\hat{A},R^q\hat{S}(\Omega)\otimes L^\vee)\to
H^p(\hat{A},R^q\hat{S}(\Omega_e)\otimes L^\vee)$$ is the zero map.
Since $\{R^{q-1}\hat{S}(\Omega_e)\}$ satisfies the Mittag-Leffler
condition, by Theorem \ref{commute_projlim}, we have
$$R^q\hat{S}(\Omega)\cong \varprojlim R^q\hat{S}(\Omega_e).$$
Then by the Mittag-Leffler condition of $\{R^q\hat{S}(\Omega_e)\}$,
we have the following isomorphism
$$H^p(\hat{A},R^q\hat{S}(\Omega)\otimes
L^\vee)\cong H^p(\hat{A},\varprojlim R^q\hat{S}(\Omega_e)\otimes
L^\vee)\cong \varprojlim H^p(\hat{A},R^q\hat{S}(\Omega_e)\otimes
L^\vee).$$ Combining with the zero map above, we obtain that the
natural maps $$\varprojlim H^p(\hat{A},R^q\hat{S}(\Omega_e)\otimes
L^\vee) \to H^p(\hat{A},R^q\hat{S}(\Omega_e)\otimes L^\vee)$$ are
all zero for any $e$ sufficiently large. By the universal property
of inverse limits, we conclude that
$$H^p(\hat{A},R^q\hat{S}(\Omega)\otimes L^\vee)\cong
\varprojlim H^p(\hat{A},R^q\hat{S}(\Omega_e)\otimes L^\vee)=0.$$

Consider the following spectral sequence
$$H^p(\hat{A},R^q\hat{S}(\Omega)\otimes L^\vee)\Rightarrow H^{p+q}(\hat{A},R\hat{S}(\Omega)\otimes
L^\vee).$$ By the discussion above,
$H^p(\hat{A},R^q\hat{S}(\Omega)\otimes L^\vee)=0$ if $p+q>g$. Hence,
$$H^l(\hat{A},R\hat{S}(\Omega)\otimes L^\vee)=0$$ for any $l>g$.
We apply Theorem \ref{commute_projlim} with
$T=\Gamma(\hat{S}(\bullet)\otimes L^\vee)$ and get
$$\varprojlim H^l(\hat{A},R\hat{S}(\Omega_e)\otimes L^\vee)=H^l(\hat{A},R\hat{S}(\Omega)\otimes
L^\vee)=0.$$

Since $\Omega_e$ and $L^\vee$ are both coherent, we may apply
Proposition \ref{FM_vs_projformula} and obtain
$$\begin{array}{r@{\;\cong\;}l}
H^l(\hat{A},R\hat{S}(\Omega_e)\otimes L^\vee) &
H^l(A,\Omega_e\otimes RS(L^\vee))\\
& H^l(A,\Omega_e\otimes RS(D_{\hat{A}}(L)[-g])) \\
& H^{l-g}(A,\Omega_e\otimes RS(D_{\hat{A}}(L))) \\
& H^{l-g}(A,\Omega_e\otimes (-1_A)^*D_A(RS(L))[-g]) \\
& H^{l-g}(A,\Omega_e\otimes (-1_A)^*\hat{L}^\vee),
\end{array}$$
where the fourth isomorphism is by Proposition \ref{FM_vs_dual}.
Taking the inverse limit, we have
$$0=\varprojlim H^l(\hat{A},R\hat{S}(\Omega_e)\otimes L^\vee)\cong
\varprojlim H^{l-g}(A,\Omega_e\otimes (-1_A)^*\hat{L}^\vee)\cong
H^{l-g}(A,\Omega\otimes (-1_A)^*\hat{L}^\vee),$$ for any $l>g$. The
theorem follows.

\end{proof}

\section{Applications}

Let $A$ be an abelian variety of dimension $g$. We say that a sheaf
$\mathcal{F}$ on $A$ satisfies IT$_i$ for some $0\leqslant
i\leqslant g$ if $H^j(A,\mathcal{F}\otimes P_\alpha)=0$ for any
$j\neq i$ and $P_\alpha\in \Pic^0(A)$. For example, if $H$ is an
ample line bundle, then $H$ satisfies IT$_0$.

We prove the following preservation of vanishing as a generalization
of \cite[Proposition 3.1 and Theorem 3.2]{PPIII}.

\begin{proposition} \label{tensor_IT0}
Let $\{\Omega_e\}$ be a GV-inverse system of coherent sheaves. Let
$H$ be a locally free sheaf satisfying IT$_0$. Then $\Omega\otimes
H$ satisfies IT$_0$.
\end{proposition}

\begin{proof}
Consider any $\alpha\in \Pic^0(A)$. Since $H$ satisfies IT$_0$, it
follows that $R\hat{S}(H\otimes \alpha)=R^0\hat{S}(H\otimes \alpha)$
is a locally free sheaf on $\hat{A}$. Since $\{\Omega_e\}$ satisfies
the Mittag-Leffler condition, we have
$$\begin{array}{r@{\;\cong\;}l}
H^i(A,\Omega\otimes H\otimes \alpha) & \varprojlim
H^i(A,\Omega_e\otimes
H\otimes \alpha) \\
& D_k (\varinjlim D_k (H^i(A,\Omega_e\otimes H\otimes \alpha)))\\
& D_k (\varinjlim H^{-i}(A,D_A(\Omega_e)\otimes (H\otimes
\alpha)^\vee)) \\
& D_k (H^{-i}(A,\hocolim D_A(\Omega_e)\otimes
(H\otimes \alpha)^\vee)) \\
& D_k (\Ext^{-i}(H\otimes \alpha,\hocolim
D_A(\Omega_e))) \\
& D_k (\Ext^{-i}(R^0\hat{S}(H\otimes
\alpha),\mathcal{H}^0(\Lambda))) \\
& D_k (H^{-i}(\hat{A},\mathcal{H}^0(\Lambda)\otimes
(R^0\hat{S}(H\otimes
\alpha))^\vee)) \\
& 0,
\end{array}$$
when $i>0$, where the sixth isomorphism is by Proposition
\ref{FM_vs_Ext}.
\end{proof}

\begin{proposition}
Let $\{\Omega_e\}$ be a GV-inverse system of coherent sheaves. Let
$\mathcal{E}$ be a locally free coherent sheaf satisfying the
Generic Vanishing conditions in the sense of \cite{PP11}. Then
$\{\Omega_e\otimes \mathcal{E}\}$ is a GV-inverse system.
\end{proposition}

\begin{proof}
Let $L$ be a sufficiently ample line bundle on $\hat{A}$. Since
$\mathcal{E}$ satisfies the Generic Vanishing conditions, by
\cite[Theorem 2.3(2)]{PPIII}, $\mathcal{E}\otimes \hat{A}^\vee$
satisfies IT$_0$. By Proposition \ref{tensor_IT0}, $\Omega\otimes
\mathcal{E}\otimes \hat{A}^\vee$ also satisfies IT$_0$. In
particular, $H^i(A, \Omega\otimes \mathcal{E}\otimes
\hat{A}^\vee)=0$. The proposition follows.
\end{proof}

\end{document}